\documentclass{amsart}
\usepackage{enumerate}
\usepackage{tikz}
\usepackage{amssymb}
\usepackage{xcolor}

\title[Edge-choosability in simple graphs with restricted odd cycles] {The list chromatic index of simple graphs whose odd cycles intersect in at most one edge}
\author{Jessica McDonald}
\author{Gregory J.~Puleo}
\thanks{The first author is supported in part by NSF grant DMS-1600551}
\address{Department of Mathematics and Statistics, Auburn University, Auburn, Alabama, USA 36849}
\email[Jessica McDonald]{mcdonald@auburn.edu}
\email[Gregory J. Puleo]{gjp0007@auburn.edu}
\usetikzlibrary{arrows}
\usetikzlibrary{decorations.pathreplacing}
\usetikzlibrary{decorations.pathmorphing}
\tikzstyle{vertex}=[inner sep = 0pt, minimum width=4pt, fill=black, shape=circle]
\tikzstyle{squarevert}=[inner sep = 0pt, minimum width=4pt, minimum height=4pt, fill=white, shape=rectangle, draw=black, thick]
\tikzstyle{trivert}=[inner sep = 0pt, minimum width=6pt, minimum height=6pt, fill=gray!50!white, shape=regular polygon,regular polygon sides=3, draw=black, thick]
\tikzstyle{pentvert}=[inner sep = 0pt, minimum width=6pt, minimum height=6pt, fill=gray!15!white, shape=regular polygon,regular polygon sides=5, draw=black, thick]
\tikzstyle{stretchy}=[decorate, decoration={snake, segment length=.3cm, amplitude=.05cm}]
\newcommand{\ato}{\mathrel{\rightarrow}}
\newcommand{\bidir}{\mathrel{\leftrightarrow}}
\newcommand{\cbar}{\overline{C}}
\newcommand{\gpoint}[2]{\node[style=vertex, label=#1:$#2$]}

\newcommand{\bpoint}[1]{\gpoint{below}{#1}}
\newcommand{\apoint}[1]{\gpoint{above}{#1}}
\newcommand{\lpoint}[1]{\gpoint{left}{#1}}
\newcommand{\rpoint}[1]{\gpoint{right}{#1}}

\newcommand{\lz}{\ell}

\newcommand{\lc}{\chi_l}

\newcommand{\nats}{\mathbb{N}}

\renewcommand{\subset}{\subseteq}
\newcommand{\gee}[1]{\mathcal{G}_{#1}}
\newcommand{\gstar}{\mathcal{G}^*}
\newcommand{\st}{\colon\,}
\newcommand{\join}{\vee}
\newcommand{\iso}{\cong}

\newcommand{\newstuff}[1]{#1}

\newcommand{\caze}[2]{\textbf{Case {#1}:} \textit{#2}}
\newcommand{\sizeof}[1]{\left\lvert{#1}\right\rvert}

\newtheorem{proposition}{Proposition}[section]

\newtheorem{lemma}[proposition]{Lemma}
\newtheorem{theorem}[proposition]{Theorem}
\newtheorem{corollary}[proposition]{Corollary}
\theoremstyle{definition}

\theoremstyle{remark}
\newtheorem*{remark}{Remark}

\newcommand{\galvin}[1]{$#1$-edge-orientable}
\newcommand{\dalvin}[1]{$#1$-kernel-perfect}
\newcommand{\sgalvin}[1]{strongly \galvin{#1}}
\begin{document}
\begin{abstract}
  We study the class of simple graphs $\gstar$ for which every pair of
  distinct odd cycles intersect in at most one edge. We give a
  structural characterization of the graphs in $\gstar$ and prove that
  every $G\in \gstar$ satisfies the list-edge-coloring
  conjecture. When $\Delta(G)\geq 4$, we in fact prove a stronger
  result about kernel-perfect orientations in $L(G)$ which implies that $G$ is
  $(m\Delta(G):m)$-edge-choosable and $\Delta(G)$-edge-paintable for
  every $m \geq 1$.
\end{abstract}
\maketitle
\section{Introduction}
In this paper all graphs are assumed to be simple unless otherwise indicated.

A fundamental characterization of bipartite graphs, proved by K\"onig
\cite{konig}, is that a graph is bipartite if and only if it contains
no odd cycle. Hsu, Ikura, and Nemhauser~\cite{HIN}, and independently
Maffray~\cite{maffray}, generalized this result, giving the following
structural characterization of the class $\gee{1}$ of graphs
containing no odd cycles of length longer than 3. Here, a \emph{block}
of a graph is a \newstuff{maximal connected subgraph of $G$ having no cut
vertex}, and the \emph{join} $G \join H$ of two graphs $G$ and $H$ is
the graph obtained from their disjoint union by adding all edges
between vertices of $G$ and vertices of $H$.

\begin{theorem}[Hsu--Ikura--Nemhauser~\cite{HIN}, Maffray~\cite{maffray}]\label{thm:hin}
  A graph $G$ lies in $\gee{1}$ if and only each of its blocks $B$ satisfies
  one of the following conditions:
  \begin{itemize}
  \item $B$ is bipartite, or
  \item $B \iso K_4$, or
  \item $B \iso K_2 \join \overline{K_r}$ for some $r \geq 1$.
  \end{itemize}
\end{theorem}

In this paper we study graphs where some longer odd cycles are allowed. Let $\gstar$ be the class of graphs $G$ in which odd cycles intersect in at most one edge, i.e., for any distinct odd cycles $C_1, C_2$ in $G$, we have
$\sizeof{E(C_1) \cap E(C_2)} \leq 1$.  Since any two distinct
triangles in a graph intersect in at most one edge, we immediately
have that $\gee{1}\subseteq \gstar$. Building on Theorem \ref{thm:hin}, our first result is the following structural
characterization of the graphs in $\gstar$. Here, for positive integers $p_1, \ldots, p_k$, the $\Theta$-graph
  $\Theta_{p_1, \ldots, p_k}$ is the graph obtained from a pair of
  vertices $\{x_1,x_2\}$ joined by $k$ internally disjoint paths, with
  the $i$th path containing $p_i$ edges. (In particular, if some
  $p_i = 1$ then the corresponding path is just an edge joining $x_1$
  and $x_2$.) Figure~\ref{fig:theta} shows $\Theta_{1,2,4}$.

\begin{figure}
  \centering
  \begin{tikzpicture}
    \lpoint{x_1} (u) at (0cm, 0cm) {};
    \rpoint{x_2} (v) at (4cm, 0cm) {};
    \draw (u) -- (v);
    \apoint{} (a1) at (2cm, -1cm) {};
    \apoint{} (b1) at (1cm, 1cm) {};
    \apoint{} (b2) at (2cm, 1cm) {};
    \apoint{} (b3) at (3cm, 1cm) {};
    \draw (u) -- (a1) -- (v);
    \draw (u) -- (b1) -- (b2) -- (b3) -- (v);
  \end{tikzpicture}
  \caption{The graph $\Theta_{1,2,4}$.}
  \label{fig:theta}
\end{figure}
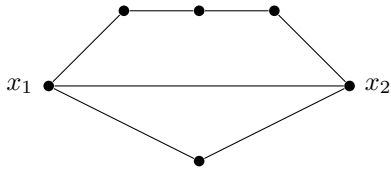
\begin{theorem}\label{thm:main}
  A graph $G$ lies in $\gstar$ if and only if each of its blocks $B$ satisfies
  one of the following conditions:
  \begin{itemize}
  \item $B$ is bipartite, or
  \item $B \iso K_4$, or
  \item $B \iso \Theta_{1, p_1, \ldots, p_r}$, where $p_1, \ldots, p_r$ are even.
  \end{itemize}
\end{theorem}

Since $K_2 \join \overline{K_r} \iso \Theta_{1, 2,\ldots,2}$, with the
$2$ repeated $r$ times, we see that the blocks permitted in
Theorem~\ref{thm:main} generalize the blocks permitted in
Theorem~\ref{thm:hin}, as required by the inclusion $\gee{1}\subseteq \gstar$. We prove this characterization in Section~\ref{sec:main}.

We use our structural characterization of the graphs in $\gstar$ to
prove a result about the list chromatic index of these graphs.  A
\emph{list assignment} to the edges of a graph $G$ is a function $\lz$
that assigns each edge $e \in E(G)$ a list of colors
$\lz(v)$. (Typically one uses the letter $L$ to denote a list
assignment; we use $\ell$ to avoid conflict with the notation $L(G)$
for line graphs.) An \emph{$\lz$-edge-coloring} of $G$ is a function
$\phi$ defined on $E(G)$ such that $\phi(e) \in \lz(e)$ for all $e$
and such that $\phi(e_1) \neq \phi(e_2)$ whenever $e_1,e_2$ are
adjacent.  A \emph{$k$-edge-coloring} of $G$ is an $\lz$-edge-coloring
for the list assignment with $\lz(e) = \{1, \ldots, k\}$ for all
$e \in E(G)$; the \emph{chromatic index} of $G$, written $\chi'(G)$,
is the smallest nonnegative $k$ such that $G$ admits a
$k$-edge-coloring.  If $f : E(G) \to \nats$ and $G$ has a proper
$\lz$-edge-coloring whenever $\sizeof{\lz(e)} \geq f(e)$ for all $e$,
we say that $G$ is \emph{$f$-edge-choosable}. In particular, if $G$ is
$f$-edge-choosable when $f(e) = k$ for all $e$, we say that $G$ is
\emph{$k$-edge-choosable}. The \emph{list chromatic index} of $G$,
written $\lc'(G)$, is the smallest nonnegative $k$ such that $G$ is
$k$-edge-choosable.

It is clear that $\chi'_l(G)\geq \chi'(G)$ for every graph $G$. The
list-edge-coloring conjecture (attributed to many sources, some as
early as 1975; see \cite{toft-jensen}) asserts that equality always
holds. In a breakthrough result in 1995, Galvin \cite{galvin} proved
the conjecture for bipartite graphs.  Peterson and Woodall
\cite{PW1,PW2} later extended this to the class $\gee{1}$.
Here we prove, in Section 3, that the list-edge-coloring conjecture
holds for $\gstar$.
\begin{theorem}\label{thm:bigmain} If $G\in\gstar$, then $\chi'_l(G)=\chi'(G)$.
\end{theorem}

As $\gee{1}\subseteq \gstar$, Theorem~\ref{thm:bigmain} extends the
work of the above-mentioned authors, in particular by allowing odd
cycles of any length. However, while our proof is limited to simple graphs, Galvin's proof also holds for bipartite multigraphs. Peterson and
Woodall's proof works for multigraphs too: they proved that $\chi'_l(G)=\chi'(G)$ for all multigraphs $G$ whose underlying simple graph is in $\gee{1}$. It would be desirable to extend Theorem \ref{thm:bigmain} to all multigraphs $G$ whose underlying simple graph is in $\gstar$, but we have not been able to do so.

In looking at Theorem \ref{thm:bigmain}, recall that Vizing's Theorem
\cite{vizing-thm} says that every graph $G$ has chromatic index
$\Delta(G)$ or $\Delta(G)+1$. We shall see that every connected graph
in $\gstar$, aside from odd cycles, satisfies
$\chi'_l(G)=\chi'(G)=\Delta(G)$. When $\Delta(G)\geq 4$, we actually
prove something stronger than the fact that $G\in\gstar$ is
$\Delta(G)$-edge-choosable: we prove a result about orienting $L(G)$
which implies results for two generalized notions of choosability. Before
describing these, we first
discuss the connection between coloring and kernels in digraphs.

An \emph{orientation} of a graph $G$ is any digraph obtained by
replacing each edge $uv \in E(G)$ with the arc $(u,v)$, the arc
$(v,u)$, or both of these arcs. A \emph{kernel} in a digraph $D$ is an
independent set of vertices $S$ such that every vertex in $D-S$ has an
out-neighbor in $S$. A digraph $D$ is said to be \emph{kernel-perfect}
if every induced subdigraph of $D$, including $D$ itself, has a
kernel. Coloring and kernels are linked by the following lemma of
Bondy, Boppana and Siegel (see \cite{galvin}); here we state the lemma
for line graphs (edge-coloring) only.

\begin{lemma}[Bondy--Boppana--Siegel]\label{lem:BBS}
  If $D$ is a kernel-perfect orientation of a line graph $L(G)$ and $f(e) \geq 1 + d_D^+(e)$
  for all $e \in E(G)$, then $G$ is $f$-edge-choosable.
\end{lemma}
\begin{remark}
  A word of caution is in order regarding our use of the word
  ``orientation''. In an orientation of a line graph $L(G)$, we
  explicitly allow the possiblity that both of the arcs $(e,f)$ and
  $(f,e)$ are present, even when $e$ and $f$ are not parallel
  edges. This possibility is also allowed by Maffray~\cite{maffray}.
  However, some papers in the literature implicitly \emph{forbid} such
  $2$-cycles, such as the paper of Borodin, Kostochka, and
  Woodall~\cite{bkw} generalizing results of \cite{maffray}.
  We mention the distinction here in the hope of avoiding future
  confusion about which notion of ``orientation'' is used in this paper.
\end{remark}

Let $H$ be a graph, and let $f : V(H) \to \nats$. We say an
orientation $D$ of $H$ is \emph{f-kernel-perfect} if it is kernel-perfect
and $f(v) \geq 1 + d_D^+(v)$ for all $v \in V(H)$; if $f(v)=k$ for all
$v$ then we say \emph{k-kernel-perfect}. We say a graph $G$ is
\emph{f-edge-orientable} if $L(G)$ admits an $f$-kernel-perfect
orientation. Using this terminology, the above lemma says that
$f$-edge-orientability implies $f$-edge-choosability. In
Section~\ref{sec:edgecol}, we prove the following result:
\begin{theorem}\label{thm:orientgstar}
  If $G \in \gstar$ has $t \geq \max\{4, \Delta(G)\}$, then
  $G$ is $t$-edge-orientable.
\end{theorem}
Proving $f$-edge-orientability of a graph also implies two properties
stronger than $f$-edge-choosability: $(mf:m)$-edge-choosability and
$f$-edge-paintability.

An \emph{$m$-tuple coloring} of a graph $G$ is a function that assigns
to each vertex $v \in V(G)$ a set $\phi(v)$ of $m$ colors such that
adjacent vertices receive disjoint sets. For $f:V(G) \to \nats$, we
say that $G$ is \emph{$(f:m)$-choosable} if $G$ has an $m$-tuple
coloring with $\phi(v) \subset \ell(v)$ whenever $\ell(v)$ is a list
assignment with $\sizeof{\ell(v)} \geq f(v)$ for all $v$. In
particular, $(f:1)$-choosability is just ordinary
$f$-choosability. This extension of list coloring was first introduced
by Erd\H os, Rubin, and Taylor in \cite{rubin} alongside
the usual notion of $f$-choosability.

A longstanding conjecture from \cite{rubin} is that every
$k$-choosable graph is $(km:m)$-choosable for all $m$;
in particular, this conjecture together with the list-edge-coloring
conjecture would imply that every graph $G$ is $(m\chi'(G):m)$-edge-choosable
for all $m \geq 1$. It was observed by Galvin~\cite{galvin} that the lemma
of Bondy, Boppana, and Siegel applies to $m$-tuple coloring as well:
in our language, if $G$ is \galvin{f}, then $G$ is $(mf:m)$-edge-choosable
for all $m \geq 1$. Thus, our Theorem~\ref{thm:orientgstar} implies that
if $G \in \gstar$ with $\Delta(G) \geq 4$, then $G$ is $(m\Delta(G):m)$-edge-choosable.

\emph{Paintability}, also known as \emph{online list coloring}, is
another extension of list coloring. Paintability was introduced by
Schauz~\cite{schauz} and independently by Zhu~\cite{zhu} and is defined
by the following $2$-player game on a graph $G$.

The online list coloring game is played by two players, Lister and
Painter, over several turns. On each turn, Lister \emph{marks} some
nonempty set $M$ of vertices, which models making the same color
available at each vertex of $M$. Then, Painter selects an independent
set $S \subset M$ and deletes all vertices of $S$ from the graph,
which models assigning the given color to the vertices of $S$ in a
proper coloring of $G$. Lister then marks some subset of the remaining
vertices, and the game continues in this manner until all vertices
have been deleted. Given a function $f : V(G) \to \nats$, the graph
$G$ is said to be \emph{$f$-paintable} if Painter has a strategy
ensuring that each vertex $v$ is marked at most $f(v)$ times before
being deleted. When $f(v) = k$ for all $v \in V(G)$, we also say
\emph{$k$-paintable}.

Given a fixed list assignment $\ell$, one strategy available to Lister
is to select, on the $i$th round, all the vertices $v$ with $i \in \ell(v)$;
if $G$ lacks an $\ell$-coloring, then this strategy wins for Lister.
Thus, every $f$-paintable graph is necessarily $f$-choosable.

Given $f : E(G) \to \nats$, we say that $G$ is
\emph{$f$-edge-paintable} if its line graph $L(G)$ is
$f$-paintable. Schauz~\cite{schauz} proved that Lemma~\ref{lem:BBS}
also works to prove $f$-edge-paintability, so that every \galvin{k}
graph is $k$-edge-paintable. Hence, Theorem~\ref{thm:orientgstar}
implies that if $G \in \gstar$ with $\Delta(G) \geq 4$, then $G$ is
$\Delta(G)$-edge-paintable.  (In fact, one can also define
$(mf:m)$-paintability in a manner analogous to $(mf:m)$-choosability,
where each vertex must be ``deleted'' $m$ times before it is actually
removed. Lemma~\ref{lem:BBS} then also shows that every \galvin{k}
graph is $(mf:m)$-edge-paintable.)

In Section~\ref{sec:sharp} we construct a subcubic graph in
$\gee{1}$ that is not $3$-edge-orientable, which shows that the
lower bound on $t$ in Theorem~\ref{thm:orientgstar} is sharp.

\section{Characterizing graphs in $\gstar$}\label{sec:main}
It is easy to see that a graph $G$ lies in $\gstar$ if and only if all
its blocks lie in $\gstar$.  Thus, in order to prove
Theorem~\ref{thm:main}, it suffices to prove the claim for
$2$-connected graphs.
\begin{lemma}
  If $G$ is $2$-connected, then $G \in \gstar$ if and only if one
  of the following conditions holds:
  \begin{itemize}
  \item $G$ is bipartite, or
  \item $G \iso K_4$, or
  \item $G \iso \Theta_{1, p_1, \ldots, p_r}$, where $p_1, \ldots, p_r$ are even.
  \end{itemize}
\end{lemma}
It is easy to check that all graphs satisfying one of these conditions
lie in $\gstar$. In the rest of this section, we show that if
$G \in \gstar$ is $2$-connected and $G$ is neither bipartite nor
isomorphic to $K_4$, then $G$ is a $\Theta$-graph of the desired form.

Fix some graph $G \in \gstar$ and let $C$ be a longest odd cycle in
$G$; we also use $C$ to refer to the vertex set of this cycle. If
$\sizeof{C} = 3$, then $G \in \gee{1}$, so Theorem~\ref{thm:hin}
immediately implies the claim.  Thus, we may assume that
$\sizeof{C} \geq 5$. Note that $G \in \gstar$ implies that $C$ has no
chords, that is, the induced subgraph $G[C]$ is a cycle.

Let $\cbar = V(G) - C$. If $\cbar = \emptyset$, then $G \iso
\Theta_{1, \sizeof{C}-1}$ and the claim holds, so assume that $\cbar
\neq \emptyset$. For $v \in \cbar$ and $w \in C$, say that $v$
\emph{touches} $w$ if there is a $v,w$-path that intersects $C$ only
at $w$. Let $T(v) = \{w \in C \st \text{$v$ touches $w$}\}$.  Since
$G$ is $2$-connected, we have $\sizeof{T(v)} \geq 2$ for all $v \in
\cbar$.

For $x,y \in C$, an \emph{external $x,y$-path} is an $x,y$-path $P$ with at least two edges
that meets $C$ only at its endpoints.
\begin{lemma}\label{lem:external-path}
  Let $v \in \cbar$ and let $H$ be the component of $G[\cbar]$
  containing $v$. If $x,y \in T(v)$, then there is an external
  $x,y$-path $P$ such that the internal vertices of $P$ all lie in
  $H$.
\end{lemma}
\begin{proof}
  Let $P_x$ be a $v,x$-path and let $P_y$ be a $v,y$-path such that
  $P_x$ and $P_y$ each meet $C$ only at their endpoints, as guaranteed
  by the definition of $T(v)$. Now $P_x \cup P_y$ is a connected
  subgraph of $G$ containing both $x$ and $y$, and the only
  $C$-vertices contained in $P_x \cup P_y$ are $x$ and $y$, with all
  other vertices in $H$. Hence there is an $x,y$-path in
  $P_x \cup P_y$, and such a path is necessarily an external
  $x,y$-path with all internal vertices in $H$.
\end{proof}
\begin{lemma}\label{lem:tclique}
  For all $v \in \cbar$, the set $T(v)$ is a pair of adjacent
  vertices in $C$.
\end{lemma}
\begin{proof}
  First we argue that $T(v)$ is a clique.
  Suppose to the contrary that $T(v)$ contains nonadjacent vertices
  $w, z \in C$. Let $Q$ be an external $w,z$-path, as guaranteed by
  Lemma~\ref{lem:external-path}. Taking $P_1$ and $P_2$ to be the two
  internally-disjoint $w,z$-paths in $C$, we see that either
  $Q \cup P_1$ or $Q \cup P_2$ is an odd cycle, as illustrated in
  Figure~\ref{fig:q}. Furthermore, both $Q \cup P_1$ and $Q \cup P_2$
  intersect $C$ in at least two edges, since $w$ and $z$ are
  nonadjacent. Thus, if some $T(v)$ is not a clique then
  $G \notin \gstar$.
  \begin{figure}
    \centering
    \begin{tabular}{cc}
    \begin{tikzpicture}
      \foreach \i in {1,...,5}
      {
        \apoint{} (c\i) at (-90+72*\i : 1cm) {};
      }
      \draw (c1) -- (c2) -- (c3) -- (c4) -- (c5) -- (c1);
      \lpoint{w} (w) at (-90-72 : 1cm) {};
      \rpoint{z} (z) at (-90+72 : 1cm) {}; 
      \draw[very thick] (c1) -- (c5) -- (c4);
      \draw[very thick, stretchy] (w) .. controls ++(-90:2cm) and ++(-90:2cm) .. (z);
      \node at (.2cm, -2.2cm) {$Q$};
    \end{tikzpicture}&
    \begin{tikzpicture}
      \foreach \i in {1,...,5}
      {
        \apoint{} (c\i) at (-90+72*\i : 1cm) {};
      }
      \draw (c1) -- (c2) -- (c3) -- (c4) -- (c5) -- (c1);
      \lpoint{w} (w) at (-90-72 : 1cm) {};
      \rpoint{z} (z) at (-90+72 : 1cm) {}; 
      \draw[very thick] (c1) -- (c2) -- (c3) -- (c4);
      \draw[very thick, stretchy] (w) .. controls ++(-90:2cm) and ++(-90:2cm) .. (z);
      \node at (.2cm, -2.2cm) {$Q$};
    \end{tikzpicture}
    \end{tabular}
    \caption{Two cycles obtained when $T(v)$ is not a clique, for $v \in \cbar$. Wavy lines represent paths;
      thick lines denote cycle edges.}
    \label{fig:q}
  \end{figure}
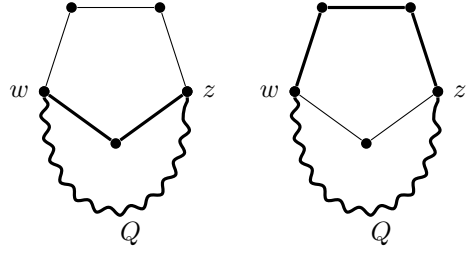

  Now, since $T(v)$ is a clique and, by definition, $T(v) \subset C$ where
  $C$ is a chordless cycle, it follows that $\sizeof{T(v)} \leq 2$.
  Since we also have $\sizeof{T(v)} \geq 2$ since $G$ is $2$-connected,
  the conclusion follows.
\end{proof}
\begin{lemma}\label{lem:ext-even}
  If $xy \in E(C)$ and $P$ is an external $x,y$-path, then $P$ has
  an even number of edges.
\end{lemma}
\begin{proof}
  If not, then $(C-xy) \cup P$ is an odd cycle that intersects
  $C$ in at least $2$ edges.
\end{proof}
\begin{lemma}\label{lem:constant}
  For all $v,w \in \cbar$, we have $T(v) = T(w)$.
\end{lemma}
\begin{proof}
  If $v,w$ are in the same component of $G[\cbar]$, this is
  clear. Otherwise, let $v$ and $w$ be vertices in different
  components of $G[\cbar]$ such that $T(v) \neq T(w)$. By
  Lemma~\ref{lem:tclique}, we may write $T(v) = \{x,y\}$ and
  $T(w) = \{s,t\}$. Let $P$ be an external $x,y$-path and let $Q$ be
  an external $s,t$-path as guaranteed by
  Lemma~\ref{lem:external-path}.  Every vertex of $V(P) \cap \cbar$ is
  in the same component of $G[\cbar]$ as $v$, while every vertex of
  $V(Q) \cap \cbar$ is in the same component of $G[\cbar]$ as $w$. It
  follows that $P$ and $Q$ are internally disjoint.

  By Lemma~\ref{lem:tclique}, we have $\{xy, st\} \subset E(C)$, so by
  Lemma~\ref{lem:ext-even}, the paths $P$ and $Q$ each have an even
  number of edges. Thus $(C - \{xy, st\}) \cup (P \cup Q)$ is an odd
  cycle that shares at least $\sizeof{C} - 2$ edges with $C$.  As
  $\sizeof{C} \geq 5$, this contradicts $G \in \gstar$.
\end{proof}
Since Lemma~\ref{lem:constant} states that $T$ is constant on $\cbar$,
in the rest of the proof we let $\{x_1, x_2\}$ be the constant value
of $T$. By Lemma~\ref{lem:tclique}, we have $x_1x_2 \in E(C)$. Our
goal now is to prove that every component of $G[\cbar]$ is a path with
one endpoint adjacent to $x_1$, the other endpoint adjacent to $x_2$,
and all internal vertices nonadjacent to $C$. As such a path must have
an even number of edges, this clearly implies that $G$ is a
$\Theta$-graph having the desired form. Let $H$ be any component of
$G[\cbar]$. For $v \in V(G)$, let $N_H(v) = N(v) \cap V(H)$.
\begin{lemma}\label{lem:unique-path}
  If $v \in N_H(x_1)$, then $H$ has exactly one path $P$ such
  that $P$ is a $v,w$-path with $w \in N_H(x_2)$. By symmetry, the
  same statement holds when $x_1, x_2$ are interchanged.
\end{lemma}
\begin{proof}
  Since $x_2 \in T(v)$, it is clear that at least one such path
  exists, so we must show that it is unique. Suppose to the contrary
  that $P_1$ and $P_2$ are distinct paths in $H$ such that $P_i$ is a
  $v,w_i$-path with $w_i \in N(x_2)$. (Possibly $w_1 = w_2$.)  Let
  $Q_i$ be the path $x_1P_ix_2$, so that $Q_1$ and $Q_2$ are distinct
  external $x_1,x_2$-paths. By Lemma~\ref{lem:ext-even}, the paths
  $Q_1$ and $Q_2$ each have an even number of edges, so that
  $Q_1 + x_1x_2$ and $Q_2 + x_1x_2$ are distinct odd cycles whose
  intersection contains $\{x_1x_2, x_1v\}$.  This contradicts
  $G \in \gstar$.
\end{proof}
\begin{corollary}
  $\sizeof{N_H(x_1)} = \sizeof{N_H(x_2)} = 1$.
\end{corollary}
\begin{proof}
  Since every vertex of $H$ touches $x_1$ and $x_2$, each neighborhood
  is nonempty. If $\sizeof{N_H(x_1)} > 1$, say, then taking
  $v \in N_H(x_2)$ and distinct $w_1,w_2 \in N_H(x_1)$, the
  connectedness of $H$ implies that there is a $v, w_i$-path in $H$ for each
  $i$. Since these paths have different endpoints, they are distinct,
  contradicting Lemma~\ref{lem:unique-path}.
\end{proof}
\begin{lemma}\label{lem:path-cpt}
  Let $v_1,v_2$ be the unique neighbors of $x_1,x_2$ in $H$, respectively.
  The component $H$ is a $v_1,v_2$-path.
\end{lemma}
\begin{proof}
  Let $P$ be the unique $v_1,v_2$-path in $H$, as guaranteed by
  Lemma~\ref{lem:unique-path}. The uniqueness of $P$ immediately implies
  that $P$ is chordless, i.e., $H[P] = P$. We prove that $V(H) = V(P)$.

  \newstuff{Suppose $w \in V(H) - V(P)$. Since $G$ is $2$-connected, there
  are two internally disjoint paths $P_1$ and $P_2$ such that $P_1$
  is a $w,v_1$-path and $P_2$ is a $w, v_2$-path. All vertices of $P_1$
  and $P_2$ share the same component of $G[\overline{C}]$ as $w$, so
  $V(P_1) \cup V(P_2) \subset H$. Letting $P_1^{-1}$ be the path
  $P_1$ traversed from $v_1$ to $w$, we see that $P_1^{-1}P_2$ is
  a $v_1,v_2$-path. This path contains $w$, and is therefore distinct
  from $P$, contradicting Lemma~\ref{lem:unique-path}.}
\end{proof}
Lemma~\ref{lem:path-cpt} completes the proof of Theorem~\ref{thm:main}.

\section{Proof of Theorems \ref{thm:bigmain} and \ref{thm:orientgstar}}\label{sec:edgecol}

Peterson and Woodall \cite{PW1,PW2} proved that a graph $G\in \gee{1}$
is $\Delta(G)$-edge-choosable by appealing to the characterization of
Theorem \ref{thm:hin}, and defining a locally-stronger notion of
edge-choosability that allowed them to prove their result block by
block. We extend their approach, defining a locally-stronger notion of
edge-orientability, and show that if this stronger definition holds
for each block of $G$, then $G$ itself is $k$-edge-orientable.

A graph $G$ is \emph{strongly k-edge-orientable} if for every $v \in V(G)$, the graph $G$ is \galvin{f_{k,v}}, where $f_{k,v}$ is defined by
  \[ f_{k,v}(e) =
  \begin{cases}
    d(v), &\text{ if $e$ is incident to $v$,} \\
    k,&\text{ otherwise.}
  \end{cases} \]
for all $e\in E(G)$. We call this a strengthening of $k$-edge-orientability (defined in the introduction) because we are always interested in $k\geq \Delta(G)$. (Recall that $\chi'_l(G)\geq \chi'(G)\geq \Delta(G)$ for every $G$, with $k$-edge-orientability implying $k$-edge-choosability.)

When discussing an orientation $D$ of some line graph $L(G)$, we will
write $e \ato f$ both to stand for the arc $(e,f)$ and as shorthand
for the statement that $(e,f) \in E(D)$. We also write $e \bidir f$ as
shorthand for ``$e \ato f$ and $f \ato e$''.
\begin{lemma}\label{lem:arrow}
  Let $H$ be a graph, and let $(X,Y)$ be a partition of $V(H)$.
  Let $D_X$ and $D_Y$ be kernel-perfect orientations of $H[X]$ and $H[Y]$,
  respectively. If $D$ is the orientation of $H$ obtained from $D_X \cup D_Y$
  by orienting every edge of $[X,Y]$ from its endpoint in $Y$ to its endpoint in $X$,
  then $D$ is a kernel-perfect orientation of $H$.
\end{lemma}
\begin{proof}
  To see that $D$ is kernel-perfect, let $Z$ be any subset of $V(H)$; we
  must show that $D[Z]$ has a kernel. Let $Z_X = Z \cap X$. Since
  $D_X$ is kernel-perfect, $D[Z_X]$ has some kernel $S_X$. (Possibly
  $Z_X = \emptyset$, in which case we can take $S_X = \emptyset$.)
  Define
  \[ Z_Y = \{v \in Z \cap Y \st N(v) \cap S_X = \emptyset\}. \]
  As $D_Y$ is kernel-perfect, $D[Z_Y]$ has some kernel $S_Y$. (As before,
  if $Z_Y = \emptyset$, then we take $S_Y = \emptyset$.)

  Let $S = S_X \cup S_Y$. We claim that $S$ is a kernel for $Z$.
  Clearly $S \subset Z$, and clearly $S$ is independent, since $S_X$
  and $S_Y$ are independent, and $N(v) \cap S_X = \emptyset$ for all
  $v \in S_Y$. It remains to show that every $w \in Z-S$ has an
  out-neighbor in $S$. If $w \in X$, then we have $w \in Z_X$, so $w$
  has an out-neighbor in $S_X$, since $S_X$ is a kernel for
  $D[Z_X]$. Otherwise, $w \in Y$. If $N(w) \cap S_X \neq \emptyset$,
  then since all edges in $[X,Y]$ are oriented from their endpoint in
  $Y$ to their endpoint in $X$, we see that $w$ has an out-neighbor in
  $S_X$. On the other hand, if $N(w) \cap S_X = \emptyset$, then
  $w \in Z_Y$, so $w$ has an out-neighbor in $S_Y$, since $S_Y$ is a
  kernel for $D[Z_Y]$.  This completes the proof.
\end{proof}
\begin{lemma}\label{lem:combineblocks}
  If $k \geq \Delta(G)$ and every block of $G$ is \sgalvin{k}
  then $G$ is \sgalvin{k}.
\end{lemma}
\begin{proof}
  It suffices to show that if $V(G) = A \cup B$ where
  $A \cap B = \{z\}$, $G[A]$ and $G[B]$ are \sgalvin{k}, and
  there are no edges from $A-z$ to $B-z$, then $G$ is \sgalvin{k}.

  Let any $v \in V(G)$ be given; we must show that $G$ is
  \galvin{f_{k,v}}.  Without loss of generality, suppose that
  $v \in A$.  (It is possible that $v=z$, in which case $v \in B$ as
  well.) Let $D_A$ and $D_B$ be kernel-perfect orientations of
  $L(G[A])$ and $L(G[B])$, respectively, such that
  \begin{align*}
    d^+_{D_A}(e) + 1 &\leq k &&\text{for all $e \in E(A)$,} \\
    d^+_{D_A}(e) + 1 &\leq d_{G[A]}(v)\leq k &&\text{for all $e \in E(A)$ with $v \in e$,} \\
    d^+_{D_B}(e) + 1 &\leq k &&\text{for all $e \in E(B)$, and} \\
    d^+_{D_B}(e) + 1 &\leq d_{G[B]}(z)\leq k &&\text{for all $e \in E(B)$ with $z \in e$.}
  \end{align*}
  Such orientations exist because $G[A]$ and $G[B]$ are \sgalvin{k}, and since $d_{G[A]}(v), d_{G[B]}(z) \leq \Delta(G)\leq k$.

  The orientations $D_A$ and $D_B$ give a direction for every edge of
  $L(G)$ except for the edges between $E(A)$ and $E(B)$, all of which
  correspond to a pair of $G$-edges incident to $z$. Let $D$ be the
  orientation that agrees with $D_A$ on $E(A)$, agrees with $D_B$ on
  $E(B)$, and orients every edge between $E(A)$ and $E(B)$ from the
  edge in $E(B)$ to the edge in $E(A)$. By Lemma~\ref{lem:arrow}
  (applied to the graph $H = L(G)$), the orientation $D$ is
  kernel-perfect. We claim that $d_D^+(e) + 1 \leq k$ for all
  $e \in E(G)$ and that $d_D^+(e) + 1 \leq d(v)$ for all $e$ incident
  to $v$.

  To show our first claim, first observe that if $e$ is not incident to $z$, then $e \in E(A)$
  or $e \in E(B)$, so that $d_D^+(e) = d_{D_A}^+(e)$ or $d_D^+(e) = d_{D_B}^+(e)$
  as appropriate, and so $d_D^+(e) \leq k - 1$. Likewise,
  if $e \in E(A)$ and $e$ is incident to $z$, then $d_D^+(e) = d_{D_A}^+(e) \leq k-1$.
  If $e \in E(B)$ and $e$ is incident to $z$, then we have
  \[ d_D^+(e) = d_{D_B}^+(e) + d_A(z) \leq d_B(z) - 1 + d_A(z) \leq \Delta(G) - 1 \leq k-1. \]
  Thus, $d_D^+(e) \leq k - 1$ for all $e \in E(G)$.

  Now suppose that $e$ is incident to $v$. If $e \in E(A)$, then we immediately
  have $d_{D}^+(e) = d_{D_A}^+(e) \leq d(v) - 1$. If $e \in E(B)$, then we must
  have $v=z$, and so
  \[ d_{D}^+(e) = d_{D_B}^+(e) + d_A(v) \leq d_{B}(v) - 1 + d_{A}(v) = d_G(v) - 1, \]
  as desired.
\end{proof}

We now proceed to show orientability for each of the block types present in Theorem \ref{thm:main}. Two of these block types, bipartite graphs and $K_4$, lie in $\gee{1}$, and hence we can avail ourselves of the following characterization due to Maffray. A \emph{clique} in a digraph is a set of vertices such that each pair $(u,v)$ is joined either by the arc $u \ato v$, or the arc $v \ato u$, or both. A vertex $u$ is a \emph{sink} in a clique if $v \ato u$ for every other vertex $v$ in the clique (possibly also $u \ato v$). Likewise, $u$ is a \emph{source} in a clique if $u \ato v$ for every other vertex $v$ in the clique.

\begin{theorem}[Maffray \cite{maffray}]\label{thm:maffray}
  If  $G \in \gee{1}$ then an orientation of $L(G)$ is kernel-perfect
  if and only if every clique in $L(G)$ has a sink.
\end{theorem}

For bipartite graphs, the standard proof of Galvin \cite{galvin} is
all that is needed to prove the orientability lemma we need. We
briefly include the details below.

\begin{lemma}\label{lem:bipgalvin}
  Every bipartite graph $G$ is \sgalvin{\Delta(G)}.
\end{lemma}
\begin{proof}
  Let $G$ be a bipartite multigraph, let $k = \Delta(G)$, and take any
  $v \in V(G)$. We show that $G$ is $f_{k,v}$-edge-orientable, i.e.,
  that $L(G)$ admits an $f_{k,v}$-kernel-perfect orientation.

  Since $G$ is bipartite, it is $k$-edge-colorable. Fix such an edge-coloring
  $\phi$ where the edges incident to $v$ have colors
  $1, \ldots, d(v)$. Fix a bipartition $(X,Y)$ of $G$ where $v \in
  X$.
  Orient the edges of $L(G)$ according to the edge-coloring of $G$:
  orient from lower color to higher color if the common endpoint lies
  in $X$, and from higher color to lower color if the common endpoint
  lies in $Y$.

  Since $G$ is triangle-free, every clique in $L(G)$ corresponds to a
  vertex in $G$, and the incident edges have a linear order in our
  orientation. Hence every clique in $L(G)$ has a sink, and since
  $G\in\gee{1}$, Theorem \ref{thm:maffray} implies that the
  orientation is kernel-perfect.

  For any edge $e\in E(G)$, each out-neighbor of $e$ in $L(G)$
  receives a different color under $\phi$: higher colors for
  incidences in $X$, and lower colors for incidences in $Y$. Hence the
  outdegree of $e$ in our orientation is at most $k-1$.

  Now consider an edge $e\in E(G)$ that is incident to $v$. We must show
  that the outdegree of $e$ in our orientation of $L(G)$ is at most
  $d(v)$. By our choice of coloring, the color of $e$ is $i$ for some
  $i \in \{1, \ldots, d(v)\}$.  As $v \in X$, the edge $e$ has at most
  $d(v)-i$ neighbors of higher color with their common endpoint in
  $X$, and at most $i-1$ neighbors of lower color with their common
  endpoint in $Y$, yielding a total of at most $d(v)-1$ out-neighbors,
  as desired.
\end{proof}

\begin{lemma}
  $K_4$ is \sgalvin{4}.
\end{lemma}
\begin{proof}
  Consider the orientation of $L(K_4)$ shown in
  Figure~\ref{fig:k4orientation}. There are no induced directed triangles, and hence by Theorem~\ref{thm:maffray}, the
  orientation is kernel-perfect. The orientation has
  maximum outdegree $3$, and satisfies $d^+(e) \leq d(v)-1=2$ for $e$
  incident to $v$.
  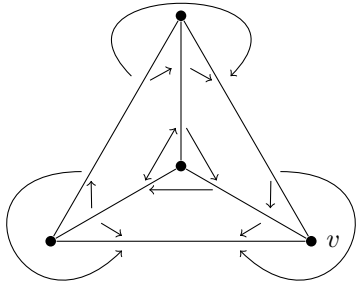
\begin{figure}
    \centering
    \begin{tikzpicture}[scale=2]
      \apoint{} (v1) at (90:1cm) {};
      \apoint{} (v2) at (210:1cm) {};
      \rpoint{v} (v3) at (330:1cm) {};
      \apoint{} (v4) at (0cm, 0cm) {};
      \draw (v1) -- (v2) node[pos=.3] (e12x) {} node[pos=.7] (e12y) {};
      \draw (v1) -- (v3) node[pos=.3] (e13x) {} node[pos=.7] (e13y) {};
      \draw (v1) -- (v4) node[pos=.3] (e14x) {} node[pos=.7] (e14y) {};
      \draw (v2) -- (v3) node[pos=.3] (e23x) {} node[pos=.7] (e23y) {};
      \draw (v2) -- (v4) node[pos=.3] (e24x) {} node[pos=.7] (e24y) {};
      \draw (v3) -- (v4) node[pos=.3] (e34x) {} node[pos=.7] (e34y) {};
      \draw[->] (e12x) -- (e14x);
      \draw[->] (e14x) -- (e13x);
      \draw[->] (e12x) .. controls ++(135:1cm) and ++(45:1cm) .. (e13x);
      \draw[->] (e13y) -- (e34x);
      \draw[->] (e34x) -- (e23y);
      \draw[->] (e13y) .. controls ++(0:1cm) and ++(-45:1cm) .. (e23y);
      \draw[->] (e14y) -- (e34y);
      \draw[->] (e34y) -- (e24y);
      \draw[<->] (e24y) -- (e14y);
      \draw[->] (e24x) -- (e12y);
      \draw[->] (e24x) -- (e23x);
      \draw[->] (e12y) .. controls ++(180:1cm) and ++(-135:1cm) .. (e23x);
    \end{tikzpicture}
    \caption{Kernel-perfect orientation of $L(K_4)$.}
    \label{fig:k4orientation}
  \end{figure}
\end{proof}

\begin{lemma}\label{lem:longcrown}
  Let $q_1 \leq \cdots \leq q_r$ be positive even integers, and let $G = \Theta_{1, q_1, \ldots, q_r}$.
  If $p \geq \max\{4, \Delta(G)\}$, then $G$ is \sgalvin{p}.
\end{lemma}
We prove Lemma~\ref{lem:longcrown} by splitting into three cases: two cases in which
$G \iso K_2 \join \overline{K_k}$ for some $k$, as allowed in Theorem~\ref{thm:hin},
and one case in which $G$ is a larger theta-graph. We state each case as its own lemma,
since we will need to refer to some of these results individually later.
\begin{lemma}\label{lem:crowngalvin-bridge}
  Let $G$ be the graph $K_2 \join \overline{K_k}$, and
  let $v$ be a vertex of maximum degree. If $p \geq \max\{3, \Delta(G)\}$,
  then $G$ is \galvin{f_{p,v}}.
\end{lemma}
\begin{proof}
  If $\Delta(G) = 2$, then $G=K_3$ and a transitive orientation of $L(G)$
  suffices (since $p\geq 3$); we just have to pick an orientation where the edge not containing $v$ is a source.
  Hence we may assume that $\Delta(G) \geq 3$, so that
  there are only two vertices of maximum degree.

  Let $u_1,u_2$ be the vertices of maximum degree, and let $G'$ be the
  graph obtained from $G$ by deleting the edge $u_1u_2$. Observe that
  $\Delta(G') = \Delta(G) - 1$. We will find an orientation of $L(G')$
  and then extend this orientation to get the desired orientation of
  $L(G)$.

  As $G'$ is bipartite, Lemma~\ref{lem:bipgalvin} implies that $L(G')$
  admits a kernel-perfect orientation $D'$ such that:
  \begin{itemize}
  \item $d^+_{D'}(e) \leq \Delta(G')-1= \Delta(G)-2$ for all $e \in E(G')$, and
  \item $d^+_{D'}(e) \leq d_{G'}(v)-1$ for all $e \in E(G')$ with $v \in e$.
  \end{itemize}
  Now we extend $D'$ to an orientation $D$ of $L(G)$. The only edges of $L(G)$ that remain to orient
  are the edges joining $u_1u_2$ with the edges of $G'$: orient all of these towards their endpoint $u_1u_2$. The
  resulting orientation is kernel-perfect, since if an induced
  subdigraph of $L(G)$ contains $u_1u_2$, the single vertex $u_1u_2$
  is a kernel; this also follows from Lemma \ref{lem:arrow}. Note that every edge aside from $u_1u_2$ (which has
  outdegree $0$) has gained outdegree exactly $1$. Since
  $d_{G'}(v)=d_G(v)-1$ , the degree bound is also satisfied.
\end{proof}
\begin{lemma}\label{lem:crowngalvin-tip}
  Let $G$ be the graph $K_2 \join \overline{K_k}$, and
  let $v$ be a vertex of degree $2$. If $p \geq \max\{4, \Delta(G)\}$,
  then $G$ is \galvin{f_{p,v}}.
\end{lemma}
\begin{proof}
  By Lemma \ref{lem:crowngalvin-bridge}, we may assume that $k > 1$ so that $G \not\iso K_3$.  Let
  $u_1,u_2$ be the vertices in the copy of $K_2$, and let $G'$ be the
  graph obtained from $G$ by deleting $u_1u_2$. As in the proof of Lemma \ref{lem:crowngalvin-bridge},
  we may take $D'$ to be a kernel-perfect orientation of $G'$ such that:
  \begin{itemize}
  \item $d^+_{D'}(e) \leq \Delta(G')-1= \Delta(G)-2$ for all $e \in E(G')$, and
  \item $d^+_{D'}(e) \leq d_{G'}(v)-1=1$ for all $e \in E(G')$ with $v \in e$.
  \end{itemize}
  Now we extend $D'$ to an orientation $D$ of $L(G)$. The only edges of $L(G)$ that remain to orient
  are the edges joining $u_1u_2$ with the edges of $G'$. Call these the \emph{undetermined edges.}

  Since $d_{D'}^+(u_i v) \leq 1$ for $i \in \{1, 2\}$, and since either $u_1v \ato u_2v$
  or $u_2v \ato u_1v$, we see that there is at most one edge $e^* \in E(G')$ with $v \notin e^*$ that satisfies $u_1v \ato e^*$ or $u_2v \ato e^*$ (and these cannot both occur).

  Now we orient the undetermined edges. If $f\neq u_1u_2$ is incident to $u_i$ for some $i \in \{1,2\}$, then:
  \begin{itemize}
  \item If $v \in f$, we orient $u_1u_2 \ato f$.
  \item If $f = e^*$, we orient $u_1u_2 \bidir f$.
  \item Otherwise, we orient $f \ato u_1u_2$.
  \end{itemize}

  Observe that $d^+_D(u_1u_2) \leq 3 \leq p-1$.  Next we consider the
  outdegree of each vertex in $D'$ as compared to its outdegree in
  $D$.  The edges incident to $v$ have not gained any more outdegree,
  so $d^+_D(e) \leq d(v) - 1$ for all edges $e$ incident to $v$.  The
  other edges have gained outdegree at most one, so
  $d_D^+(e) \leq (\Delta(G')-1)+1=\Delta(G)-1$ for such edges.

  We claim that $D$ is kernel-perfect. Since $G\in\gee{1}$, by
  Theorem \ref{thm:maffray} it again suffices to prove that every
  clique has a sink. Since $D'$ is kernel-perfect, we need only
  consider a clique of the form $K\cup\{u_1u_2\}$ where $K$ is a
  clique in $D'$; suppose $e$ is a sink of $K$ in $D'$.  If
  $u_1u_2 \ato e$, then $e$ remains a sink. So suppose that
  $u_1u_2 \not\ato e$. We claim that in this case $u_1u_2$ is the
  required sink. If not, then there exists $f$ in $K$ with
  $f\not\ato u_1u_2$. This implies that $v \in f$. On the other hand,
  since $u_1u_2 \not\ato e$, we have $v \notin e$ and $e \neq e^*$.
  In particular, since $e \neq e^*$, we have $f \not\ato e$. This
  contradicts the assumption that $e$ is a sink in $K$.
\end{proof}

\begin{lemma}\label{lem:thetagalvin}
  Let $G$ be the theta-graph $\Theta_{1, q_1, \ldots, q_r}$.
  If $\max_i q_i > 2$ and $p \geq \max\{3, \Delta(G)\}$,
  then $G$ is \sgalvin{p}.
\end{lemma}
\begin{proof}
  We prove the theorem by induction on $r$, with trivial base case
  when $r=1$ and $G$ is an odd cycle (in this case $\Delta(L(G)) = 2$). We may assume
  that $q_1 \leq \cdots \leq q_r$, with $q_r > 2$.  Let any
  $v \in V(G)$ be given; we produce a \dalvin{f_{p,v}} orientation of
  $L(G)$.

  First suppose that there is some edge $e^*$ with $d_{L(G)}(e^*) = 2$
  and $v \notin e^*$. This implies that both endpoints of $e$ have
  degree $2$ in $G$. Let $G' = G-e^*$. The graph $G'$ may no longer be
  $2$-connected; its blocks consist of a $\Theta$-graph
  $\Theta_{1, q'_1, \ldots, q'_{r-1}}$ and possibly some blocks
  isomorphic to $K_2$. We will show that $G'$ is \sgalvin{p}.  By
  Lemma~\ref{lem:combineblocks} it suffices to show that the smaller
  $\Theta$-graph is \sgalvin{p}, since the $K_2$-blocks clearly are.

  \caze{1}{$r=2$.} In this case, the smaller $\Theta$-graph is
  an odd cycle, and is therefore \sgalvin{p} since $p \geq 3$.

  \caze{2}{$r \geq 3$ and $\max_i q'_i > 2$.} In this case, the induction
  hypothesis immediately implies that the smaller $\Theta$-graph is
  \sgalvin{p}.

  \caze{3}{$r \geq 3$ and all $q'_i = 2$.} In this case, the smaller
  $\Theta$-graph is $K_2 \join \overline{K_{r-1}}$, and in particular
  has maximum degree $r$. Since $p \geq \Delta(G) = r+1 \geq 4$,
  Lemmas~\ref{lem:crowngalvin-bridge} and \ref{lem:crowngalvin-tip}
  imply that the smaller $\Theta$-graph is \sgalvin{p}.
  \smallskip

  In all three cases, we have argued that $G'$ is \sgalvin{p}. Let
  $D'$ be a \dalvin{f_{p,v}}-orientation of $G'$.  We extend $D'$ to a
  \dalvin{f_{p,v}}-orientation of $G$ by orienting both $L(G)$-edges
  $e^*f$ as $e^* \ato f$. This orientation gives $e^*$ outdegree $2$,
  and it is kernel-perfect since $D'$ was kernel-perfect (by Lemma \ref{lem:arrow}). Thus, $G$ is
  \galvin{f_{p,v}}.
  \medskip

  Now suppose no such edge $e^*$ exists. Together with our earlier
  assumptions, this immediately implies that
  $q_1 = \cdots = q_{r-1} = 2$, that $q_r = 4$, and that $v$
  is the middle vertex of the defining path with $4$ edges,
  as shown in Figure~\ref{fig:theta-nostar}.
  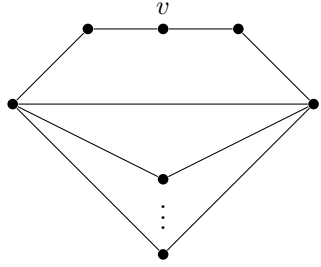
\begin{figure}
    \centering
    \begin{tikzpicture}
      \lpoint{} (u) at (0cm, 0cm) {};
      \rpoint{} (v) at (4cm, 0cm) {};
      \draw (u) -- (v);
      \apoint{} (a1) at (2cm, -1cm) {};
      \node at (2cm, -1.4cm) {$\vdots$};
      \apoint{} (a2) at (2cm, -2cm) {};
      \apoint{} (b1) at (1cm, 1cm) {};
      \apoint{v} (b2) at (2cm, 1cm) {};
      \apoint{} (b3) at (3cm, 1cm) {};
      \draw (u) -- (a1) -- (v);
      \draw (u) -- (a2) -- (v);
      \draw (u) -- (b1) -- (b2) -- (b3) -- (v);
    \end{tikzpicture}
    \caption{Case of Lemma~\ref{lem:thetagalvin} where no edge $e^*$ exists.}
    \label{fig:theta-nostar}
  \end{figure}

  Since $r \geq 2$, we have $\Delta(G) \geq 3$, so there are exactly
  two vertices of maximum degree.  Let $e$ be the edge joining the two
  vertices of maximum degree, and let $G' = G-e$. The graph $G'$ is
  bipartite, so by Lemma~\ref{lem:bipgalvin}, we can find a
  \galvin{f_{\Delta(G)-1, v}} orientation $D'$ of $L(G')$.  Extend
  $D'$ to an orientation $D$ of $G$ by making $e$ a sink; by
  Lemma~\ref{lem:arrow}, the orientation $D$ is
  kernel-perfect. Extending $D'$ to $D$ in this manner does not
  increase the outdegree of any edge incident to $v$, and every edge
  not incident to $v$ gains at most one new out-neighbor.  Since the
  maximum degree also increased by one, $G$ is \galvin{f_{p,v}}.
\end{proof}

We can now prove Theorems \ref{thm:orientgstar} and \ref{thm:bigmain}.

\begin{proof}[Proof of Theorem~\ref{thm:orientgstar}]
  By Theorem~\ref{thm:main}, every block of such a graph is either
  bipartite, isomorphic to $K_4$, or isomorphic to
  $\Theta_{1, p_1, \ldots, p_r}$ where $p_1, \ldots, p_r$ are even.
  We have proved that all such graphs are \sgalvin{t}, so by
  Lemma~\ref{lem:combineblocks}, $G$ is \sgalvin{t}.
\end{proof}

\begin{proof}[Proof of Theorem~\ref{thm:bigmain}]
  If $\Delta(G)\geq 4$ the result follows immediately from Theorem
  \ref{thm:orientgstar} and Lemma \ref{lem:BBS}, so it suffices to prove the result for
  subcubic graphs. If $G=K_4$ we refer to the reader to the proof by
  Cariolaro and Lih \cite{CL} to see that $G$ is 3-edge-choosable.
  If $\Delta(G) \leq 2$, then every component of $G$ is either a cycle or a path, and it readily follows that $G$ is $\chi'(G)$-edge-choosable (in particular a greedy approach works).

  For the remaining cases, we use induction on the number of blocks
  in $G$. Our base case (when $G$ is $2$-connected) will be handled
  as a special case of the induction step.

  If $G$ is $2$-connected (base case), then let $H = G$ and let $v$ be
  an arbitrary vertex of $G$. If $G$ is not $2$-connected (induction
  step), then let $H$ be a leaf block of $G$ and let $v$ be the cut
  vertex in $H$.

  First suppose that $H$ is \galvin{f_{3,v}}. If $G$ is $2$-connected, then
  this implies that $G=H$ and Lemma \ref{lem:BBS} implies that $G$ is $3$-edge-choosable. If $G$ is not
  $2$-connected, then let $\ell$ be any edge list assignment on $G$
  with $\sizeof{\ell(e)} \geq 3$ for all $e$. Let $G' = G - (V(H)-v)$.
  By the induction hypothesis, $G'$ is $\ell$-choosable. Taking any $\ell$-coloring
  of $G'$, we see that at most $3-d_H(v)$ colors appear on the edges of $G'$ incident to $v$.
  Deleting these colors from the $H$-edges incident to $v$ gives a list assignment
  $\ell^*$ with $\sizeof{\ell^*(e)} \geq f_{3,v}(e)$ for all $e \in E(H)$.
  Since $H$ is \galvin{f_{3,v}}, Lemma \ref{lem:BBS} guarantees that we can properly color its edges using
  the colors from $\ell^*$, allowing us to extend the $\ell$-coloring of $G'$
  to all of $G$.

  It remains to consider the case where $H$ is not \galvin{f_{3,v}}.
  By Theorem \ref{thm:main}, Lemma~\ref{lem:bipgalvin}, Lemma~\ref{lem:crowngalvin-bridge}, and Lemma~\ref{lem:thetagalvin},
  this implies that $H = K_2 \join \overline{K_2}$ and $v$ is a vertex
  of degree $2$. We will show that in this case, $H$ is $f_{3,v}$-choosable,
  even though it is not $f_{3,v}$-orientable. This follows from Theorem 2.3 of \cite{PW1},\cite{PW2}, however we include a proof here for completeness.

  Figure~\ref{fig:diamondchoose}(a)
  shows the list sizes given by $f_{3,v}$.  Let $\ell$ be any edge list
  assignment on $H$ with $\sizeof{\ell(e)} \geq f_{3,v}(e)$ for all $e$. Let $\{e_1, f_1\}$ and $\{e_2, f_2\}$ be disjoint matchings in $H$,
  with each $e_i$ incident to $v$, and let $h$ be the remaining edge
  of $H$, as shown in Figure~\ref{fig:diamondchoose}(b).

  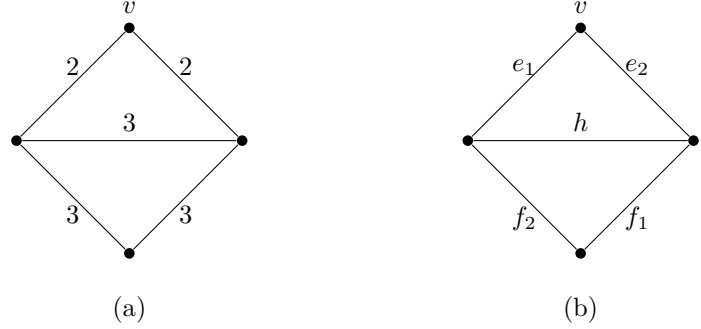
\begin{figure}
    \centering
    \begin{tikzpicture}[scale=1.5]
      \begin{scope}[xshift=-2cm]
      \apoint{v} (v) at (0cm, 1cm) {};
      \apoint{} (w) at (-1cm, 0cm) {};
      \apoint{} (z) at (1cm, 0cm) {};
      \apoint{} (u) at (0cm, -1cm) {};
      \draw (v) -- (w) node[pos=.5, above] {$2$};
      \draw (v) -- (z) node[pos=.5, above] {$2$};
      \draw (u) -- (w) node[pos=.5, below] {$3$};
      \draw (u) -- (z) node[pos=.5, below] {$3$};
      \draw (w) -- (z) node[pos=.5, above] {$3$};
      \node at (0cm, -1.5cm) {(a)};
      \end{scope}
      \begin{scope}[xshift=2cm]
      \apoint{v} (v) at (0cm, 1cm) {};
      \apoint{} (w) at (-1cm, 0cm) {};
      \apoint{} (z) at (1cm, 0cm) {};
      \apoint{} (u) at (0cm, -1cm) {};
      \draw (v) -- (w) node[pos=.5, above] {$e_1$};
      \draw (v) -- (z) node[pos=.5, above] {$e_2$};
      \draw (u) -- (w) node[pos=.5, below] {$f_2$};
      \draw (u) -- (z) node[pos=.5, below] {$f_1$};
      \draw (w) -- (z) node[pos=.5, above] {$h$};
      \node at (0cm, -1.5cm) {(b)};
      \end{scope}
    \end{tikzpicture}
    \caption{List sizes (a) and edge labels (b) in $K_2 \join \overline{K_2}$.}
    \label{fig:diamondchoose}
  \end{figure}

  First suppose that $\ell(e_i) \cap \ell(f_i) \neq \emptyset$ for
  some $i \in \{1,2\}$. Assigning the common color to both $e_i$ and $f_i$
  and deleting this color from all other lists, the remaining uncolored edges
  form a path on $3$ vertices with edge list sizes $2,2,1$. We complete the
  $\ell$-coloring by coloring this path greedily starting with the list of
  size $1$.
  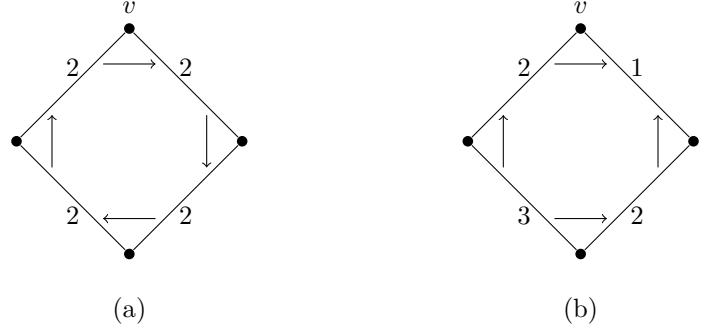
\begin{figure}
    \centering
    \begin{tikzpicture}[scale=1.5]
      \begin{scope}[xshift=-2cm]
      \apoint{v} (v) at (0cm, 1cm) {};
      \apoint{} (w) at (-1cm, 0cm) {};
      \apoint{} (z) at (1cm, 0cm) {};
      \apoint{} (u) at (0cm, -1cm) {};
      \draw (v) -- (w) node[pos=.5, above] {$2$} node[pos=.3] (vwx) {} node[pos=.7] (vwy) {};
      \draw (v) -- (z) node[pos=.5, above] {$2$} node[pos=.3] (vzx) {} node[pos=.7] (vzy) {};
      \draw (u) -- (w) node[pos=.5, below] {$2$} node[pos=.3] (uwx) {} node[pos=.7] (uwy) {};
      \draw (u) -- (z) node[pos=.5, below] {$2$} node[pos=.3] (uzx) {} node[pos=.7] (uzy) {};
      \draw[->] (vwx) -- (vzx);
      \draw[->] (vzy) -- (uzy);
      \draw[->] (uzx) -- (uwx);
      \draw[->] (uwy) -- (vwy);
      \node at (0cm, -1.5cm) {(a)};
      \end{scope}
      \begin{scope}[xshift=2cm]
      \apoint{v} (v) at (0cm, 1cm) {};
      \apoint{} (w) at (-1cm, 0cm) {};
      \apoint{} (z) at (1cm, 0cm) {};
      \apoint{} (u) at (0cm, -1cm) {};
      \draw (v) -- (w) node[pos=.5, above] {$2$} node[pos=.3] (vwx) {} node[pos=.7] (vwy) {};
      \draw (v) -- (z) node[pos=.5, above] {$1$} node[pos=.3] (vzx) {} node[pos=.7] (vzy) {};
      \draw (u) -- (w) node[pos=.5, below] {$3$} node[pos=.3] (uwx) {} node[pos=.7] (uwy) {};
      \draw (u) -- (z) node[pos=.5, below] {$2$} node[pos=.3] (uzx) {} node[pos=.7] (uzy) {};
      \draw[->] (vwx) -- (vzx);
      \draw[->] (uzy) -- (vzy);
      \draw[->] (uwx) -- (uzx);
      \draw[->] (uwy) -- (vwy);
      \node at (0cm, -1.5cm) {(b)};
      \end{scope}
    \end{tikzpicture}
    \caption{Possible remaining list sizes after coloring $h$,
with kernel-perfect orientations of $L(C_4)$.}
    \label{fig:cyclelists}
  \end{figure}

  Thus, we may assume that $\ell(e_i) \cap \ell(f_i) = \emptyset$ for $i \in \{1,2\}$.
  In this case, assign the edge $h$ a color from $\ell(h)$ not appearing in $\ell(e_1)$ (possible since $|\ell(h)|>|\ell(e_1)|$),
  and remove this color from all lists in which it appears. At worst, this removes a color
  from $\ell(f_1)$ and from at most one of $\ell(e_2)$ and $\ell(f_2)$. The two possible
  lower bounds on the remaining list sizes are shown in Figure~\ref{fig:cyclelists}.
  In both cases, the line graph of the remaining edges has a kernel-perfect orientation
  (pictured) in which every edge has outdegree $1$ less than the size of its remaining
  list. Thus, by Lemma \ref{lem:BBS}, we can complete the $\ell$-coloring.
\end{proof}

\section{Sharpness}\label{sec:sharp}

\begin{lemma}\label{lem:nok4minus}
  Let $H = K_2 \join \overline{K_2}$, and let $v$ be a vertex of degree $2$
  in $H$. There is no \dalvin{f_{3,v}}-orientation of $L(H)$.
\end{lemma}
\begin{proof}
  Suppose to the contrary that $D$ is such an orientation.  We start
  by observing that there are $8$ edges in $L(H)$ and that
  $\sum_{e \in E(H)}(f_{3,v}(e)-1) = 8$, that is, the bounds on the
  outdegree in $D$ are \emph{exactly} large enough to orient each edge
  of $L(H)$. In particular, this implies that no edge in $L(H)$ can be
  bidirected in $D$, since this would cause the total outdegree over
  all the $H$-edges to exceed $8$.

  Let $Q_1$ be the clique in $L(H)$ consisting of the two edges
  incident to $v$, and let $vw$ be a source in $Q_1$. In an
  \dalvin{f_{3,v}}-orientation of $L(H)$, we must have $d^+(vw) = 1$,
  so every $L(H)$-edge $(vw)f$ for $f \notin Q_1$ must be oriented as
  $f \ato vw$.
  \begin{figure}
    \centering
    \begin{tikzpicture}[scale=2]
      \apoint{v} (v) at (0cm, 1cm) {};
      \lpoint{w} (w) at (-1cm, 0cm) {};
      \rpoint{z} (z) at (1cm, 0cm) {};
      \bpoint{u} (u) at (0cm, -1cm) {};
      \draw (v) -- (w) node[pos=.3] (vwx) {} node[pos=.7] (vwy) {};
      \draw (v) -- (z) node[pos=.3] (vzx) {} node[pos=.7] (vzy) {};
      \draw (w) -- (z) node[pos=.3] (wzx) {} node[pos=.7] (wzy) {};
      \draw (u) -- (w) node[pos=.3] (uwx) {} node[pos=.7] (uwy) {};
      \draw (u) -- (z) node[pos=.3] (uzx) {} node[pos=.7] (uzy) {};
      \draw[->] (vwx) -- (vzx);
      \draw[->] (wzx) -- (vwy);
      \draw[->] (uwy) .. controls ++ (225:1.5cm) and ++(135:1.5cm) .. (vwy);
      \draw[->] (uzy) -- (wzy);
      \draw[->] (uzy) .. controls ++ (-45:1.5cm) and ++(45:1.5cm) .. (vzy);
      \draw[->] (uwy) -- (wzx);
      \draw[->] (wzy) -- (vzy);
    \end{tikzpicture}
    \caption{Illustration of the proof of Lemma~\ref{lem:nok4minus}.}
    \label{fig:nok4minus}
  \end{figure}
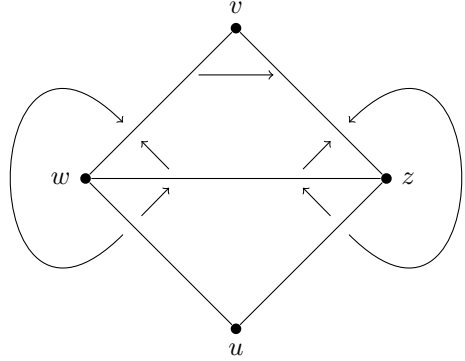
  Let $z$ be the vertex of degree $3$ distinct from $w$,
  as shown in Figure~\ref{fig:nok4minus}.

  Let $Q_2$ be the clique of $L(H)$ consisting of the edges incident
  to $z$, and let $e$ be a source in $Q_2$. Since $d^+(vz) \leq 1$
  (since we have a \dalvin{f_{3,v}}-orientation), we have $e \neq vz$.
  Furthermore, since $d^+(vw) \leq 1$ and $vw \ato vz$, we must have
  $wz \ato vz$, since otherwise we create a directed triangle in the
  line graph. Since $vw$ and $vz$ are both out-neighbors of $wz$,
  there are no other out-neighbors of $wz$. Hence, if $u$ is the final
  vertex in the graph, we know that $uw, uz \ato wz$. In particular,
  this means that $e=uz$, and hence that $uz \ato vz$. However, now
  both $uz$ and $uw$ have two out-neighbors outside $\{uw, uz\}$.
  Since also we must have either $uz \ato uw$ or $uw \ato uz$, we get
  a contradiction.
\end{proof}

\begin{theorem}
  If $G$ is the graph shown in Figure~\ref{fig:no-orn}, then $G$ is not 3-edge-orientable.
  \begin{figure}
    \centering
    \begin{tikzpicture}
      \bpoint{u} (x) at (0cm, 0cm) {};
      \foreach \i in {1,2,3}
      {
        \begin{scope}[rotate=120*\i, yshift=.5cm]
          \apoint{} (v\i) at (0cm, 1cm) {};
          \apoint{} (w\i) at (-1cm, 2cm) {};
          \apoint{} (z\i) at (1cm, 2cm) {};
          \apoint{} (u\i) at (0cm, 3cm) {};
          \draw (v\i) -- (w\i);
          \draw (v\i) -- (z\i);
          \draw (u\i) -- (w\i);
          \draw (u\i) -- (z\i);
          \draw (w\i) -- (z\i);
        \end{scope}
        \draw (x) -- (v\i);
      }
    \end{tikzpicture}
    \caption{A subcubic non-\galvin{3} graph in $\gee{1}$.}
    \label{fig:no-orn}
  \end{figure}
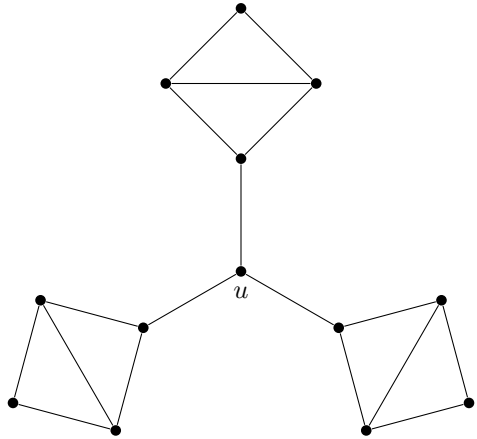
\end{theorem}
\begin{proof}
  Suppose to the contrary that $D$ is an orientation of $L(G)$ that is kernel-perfect and has outdegrees at most 2. Let $Q$ be the clique in $L(G)$ consisting of the edges incident
  to the labeled vertex $u$, let $uv$ be a source in $Q$, and
  let $H$ be the copy of $K_2 \join \overline{K_2}$ containing
  the vertex $v$. Since $d^+(uv) \leq 2$, both edges $e \in E(H)$
  incident to $v$ must have the arc $e(uv)$ oriented as $e \ato uv$.
  Hence, these edges have outdegree at most $1$ in the subdigraph
  $D[E(H)]$. Thus, $D[E(H)]$ is an \dalvin{f_{3,v}} orientation
  of $L(H)$. By Lemma~\ref{lem:nok4minus}, no such orientation of
  $L(H)$ exists, yielding a contradiction.
\end{proof}

\section{Acknowledgments}
\newstuff{The authors thank the anonymous referees for their
  careful reading of the paper and their suggestions which
  improved its presentation.}
\bibliographystyle{amsplain} \bibliography{bib-cycle}

\providecommand{\bysame}{\leavevmode\hbox to3em{\hrulefill}\thinspace}
\providecommand{\MR}{\relax\ifhmode\unskip\space\fi MR }
\providecommand{\MRhref}[2]{%
  \href{http://www.ams.org/mathscinet-getitem?mr=#1}{#2}
}
\providecommand{\href}[2]{#2}
\begin{thebibliography}{10}

\bibitem{bkw}
O.~V. Borodin, A.~V. Kostochka, and D.~R. Woodall, \emph{On kernel-perfect
  orientations of line graphs}, Discrete Math. \textbf{191} (1998), no.~1-3,
  45--49, Graph theory (Elgersburg, 1996). \MR{1644870 (99f:05043)}

\bibitem{CL}
D.~Cariolaro and K.-W. Lih, \emph{The edge-choosability of the tetrahedron},
  Math. Gazette \textbf{92} (2008), no.~525, 543--546. \MR{1694489}

\bibitem{rubin}
Paul Erd{\H{o}}s, Arthur~L. Rubin, and Herbert Taylor, \emph{Choosability in
  graphs}, Proceedings of the {W}est {C}oast {C}onference on {C}ombinatorics,
  {G}raph {T}heory and {C}omputing ({H}umboldt {S}tate {U}niv., {A}rcata,
  {C}alif., 1979) (Winnipeg, Man.), Congress. Numer., XXVI, Utilitas Math.,
  1980, pp.~125--157. \MR{593902 (82f:05038)}

\bibitem{galvin}
Fred Galvin, \emph{The list chromatic index of a bipartite multigraph}, J.
  Combin. Theory Ser. B \textbf{63} (1995), no.~1, 153--158. \MR{1309363
  (95m:05101)}

\bibitem{HIN}
Wen~Lian Hsu, Yoshiro Ikura, and George~L. Nemhauser, \emph{A polynomial
  algorithm for maximum weighted vertex packings on graphs without long odd
  cycles}, Math. Programming \textbf{20} (1981), no.~2, 225--232. \MR{607408
  (82j:68021)}

\bibitem{toft-jensen}
Tommy~R Jensen and Bjarne Toft, \emph{Graph coloring problems}, vol.~39, John
  Wiley \& Sons, 2011.

\bibitem{konig}
D.~K\"onig, \emph{Graphen und matrizen}, Mat. Lapok \textbf{38} (1931),
  116--119.

\bibitem{maffray}
Fr{\'e}d{\'e}ric Maffray, \emph{Kernels in perfect line-graphs}, J. Combin.
  Theory Ser. B \textbf{55} (1992), no.~1, 1--8. \MR{1159851 (93i:05061)}

\bibitem{PW1}
Dale Peterson and Douglas~R. Woodall, \emph{Edge-choosability in line-perfect
  multigraphs}, Discrete Math. \textbf{202} (1999), no.~1-3, 191--199.
  \MR{1694489}

\bibitem{PW2}
\bysame, \emph{Erratum: ``{E}dge-choosability in line-perfect multigraphs''
  [{D}iscrete {M}ath. {\bf 202} (1999) no. 1-3, 191--199; {MR}1694489
  (2000a:05086)]}, Discrete Math. \textbf{260} (2003), no.~1-3, 323--326.
  \MR{1948402}

\bibitem{schauz}
Uwe Schauz, \emph{Mr. {P}aint and {M}rs. {C}orrect}, Electron. J. Combin.
  \textbf{16} (2009), no.~1, Research Paper 77, 18. \MR{2515754}

\bibitem{vizing-thm}
V.~G. Vizing, \emph{On an estimate of the chromatic class of a {$p$}-graph},
  Diskret. Analiz No. \textbf{3} (1964), 25--30. \MR{0180505 (31 \#4740)}

\bibitem{zhu}
Xuding Zhu, \emph{On-line list colouring of graphs}, Electron. J. Combin.
  \textbf{16} (2009), no.~1, Research Paper 127, 16. \MR{2558264}

\end{thebibliography}
\end{document}